\pdfoutput=1
\documentclass[cleveref]{lipics-v2021}
\hideLIPIcs
\nolinenumbers

\usepackage{newunicodechar}
\usepackage{silence}
\newunicodechar{♢}{\tikz \node[inner sep=1.5,draw,diamond] {};}
\newunicodechar{☆}{\tikz \node[inner sep=1,draw,star,star point ratio=2] {};}
\newunicodechar{△}{\triangle}
\newunicodechar{⬜}{\kern 0.5pt\tikz \node[inner sep=1.7,draw,regular polygon,regular polygon sides=4] {};\kern 0.5pt}
\newunicodechar{◯}{\tikz[baseline=-3pt] \node[inner sep=1.7,draw,cloud,cloud puffs=4,cloud puff arc=190] {};}

\newunicodechar{⊥}{\bot}
\newunicodechar{•}{\item}
\newunicodechar{✓}{\checkmark}
\newunicodechar{✗}{\xmark}
\newunicodechar{…}{\dots}
\newunicodechar{≔}{\coloneqq}
\newunicodechar{⁻}{^-}
\newunicodechar{⁺}{^+}
\newunicodechar{₋}{_-}
\newunicodechar{₊}{_+}
\newunicodechar{ℓ}{\ell}
\newunicodechar{•}{\item}
\newunicodechar{…}{\dots}
\newunicodechar{≔}{\coloneqq}
\newif\ifnormopen\normopenfalse
\newunicodechar{‖}{\ifnormopen\rVert\normopenfalse\else\rVert\normopentrue\fi}
\newunicodechar{≤}{\leq}
\newunicodechar{≥}{\geq}
\newunicodechar{≰}{\nleq}
\newunicodechar{≱}{\ngeq}
\newunicodechar{⊕}{\oplus}
\newunicodechar{⊗}{\otimes}
\newunicodechar{≠}{\neq}
\newunicodechar{¬}{\neg}
\newunicodechar{≡}{\equiv}
\newunicodechar{₀}{_0}
\newunicodechar{₁}{_1}
\newunicodechar{₂}{_2}
\newunicodechar{₃}{_3}
\newunicodechar{₄}{_4}
\newunicodechar{₅}{_5}
\newunicodechar{₆}{_6}
\newunicodechar{₇}{_7}
\newunicodechar{₈}{_8}
\newunicodechar{₉}{_9}
\newunicodechar{ₚ}{_p}
\newunicodechar{ₙ}{_n}
\newunicodechar{ₐ}{_a}
\newunicodechar{ₑ}{_e}
\newunicodechar{ₕ}{_h}
\newunicodechar{ₖ}{_k}
\newunicodechar{ₗ}{_l}
\newunicodechar{ₘ}{_m}
\newunicodechar{ₛ}{_s}
\newunicodechar{ₜ}{_t}
\newunicodechar{ₓ}{_x}
\newunicodechar{⁰}{^0}
\newunicodechar{¹}{^1}
\newunicodechar{²}{^2}
\newunicodechar{³}{^3}
\newunicodechar{⁴}{^4}
\newunicodechar{⁵}{^5}
\newunicodechar{⁶}{^6}
\newunicodechar{⁷}{^7}
\newunicodechar{⁸}{^8}
\newunicodechar{⁹}{^9}
\newunicodechar{ⁿ}{^n}
\newunicodechar{∈}{\in}
\newunicodechar{∉}{\notin}
\newunicodechar{⊂}{\subset}
\newunicodechar{⊃}{\supset}
\newunicodechar{⊆}{\subseteq}
\newunicodechar{⊇}{\supseteq}
\newunicodechar{⊄}{\nsubset}
\newunicodechar{⊅}{\nsupset}
\newunicodechar{⊈}{\nsubseteq}
\newunicodechar{⊉}{\nsupseteq}
\newunicodechar{∪}{\cup}
\newunicodechar{∩}{\cap}
\newunicodechar{∀}{\forall}
\newunicodechar{∃}{\exists}
\newunicodechar{∄}{\nexists}
\newunicodechar{∨}{\vee}
\newunicodechar{∧}{\wedge}
\newunicodechar{⊼}{\bar{\wedge}}
\newunicodechar{⊽}{\bar{\vee}}
\newunicodechar{ℝ}{\mathbb{R}}
\newunicodechar{ℙ}{\mathbb{P}}
\newunicodechar{ℕ}{\mathbb{N}}
\newunicodechar{𝔼}{\mathbb{E}}
\newunicodechar{𝔽}{\mathbb{F}}
\newunicodechar{ℤ}{\mathbb{Z}}
\newunicodechar{⌊}{\lfloor}
\newunicodechar{⌋}{\rfloor}
\newunicodechar{⌈}{\lceil}
\newunicodechar{⌉}{\rceil}
\newunicodechar{·}{\cdot}
\newunicodechar{∘}{\circ}
\newunicodechar{×}{\times}
\newunicodechar{↑}{\uparrow}
\newunicodechar{↓}{\downarrow}
\newunicodechar{→}{\rightarrow}
\newunicodechar{←}{\leftarrow}
\newunicodechar{⇒}{\Rightarrow}
\newunicodechar{⇐}{\Leftarrow}
\newunicodechar{↔}{\leftrightarrow}
\newunicodechar{⇔}{\Leftrightarrow}
\newunicodechar{↦}{\mapsto}
\newunicodechar{∅}{\varnothing}
\newunicodechar{∞}{\infty}
\newunicodechar{≅}{\cong}
\newunicodechar{≈}{\approx}
\newunicodechar{ℓ}{\ell}
\newunicodechar{𝟙}{\mathds{1}}
\newunicodechar{𝟘}{\mathds{0}}
\newunicodechar{↪}{\hookrightarrow}

\newunicodechar{α}{\alpha}
\newunicodechar{β}{\beta}
\newunicodechar{γ}{\gamma}
\newunicodechar{Γ}{\Gamma}
\newunicodechar{δ}{\delta}
\newunicodechar{Δ}{\Delta}
\newunicodechar{ε}{\varepsilon}
\newunicodechar{ζ}{\zeta}
\newunicodechar{η}{\eta}
\newunicodechar{θ}{\theta}
\newunicodechar{Θ}{\Theta}
\newunicodechar{ι}{\iota}
\newunicodechar{κ}{\kappa}
\newunicodechar{λ}{\lambda}
\newunicodechar{Λ}{\Lambda}
\newunicodechar{μ}{\mu}
\newunicodechar{ν}{\nu}
\newunicodechar{ξ}{\xi}
\newunicodechar{Ξ}{\Xi}
\newunicodechar{π}{\pi}
\newunicodechar{Π}{\Pi}
\newunicodechar{ρ}{\rho}
\newunicodechar{σ}{\sigma}
\newunicodechar{Σ}{\Sigma}
\newunicodechar{τ}{\tau}
\newunicodechar{υ}{\upsilon}
\newunicodechar{ϒ}{\Upsilon}
\newunicodechar{φ}{\varphi}
\newunicodechar{ϕ}{\phi}
\newunicodechar{Φ}{\Phi}
\newunicodechar{χ}{\chi}
\newunicodechar{ψ}{\psi}
\newunicodechar{Ψ}{\Psi}
\newunicodechar{ω}{\omega}
\newunicodechar{Ω}{\Omega}

\newunicodechar{𝒜}{\mathcal{A}}
\newunicodechar{ℬ}{\mathcal{B}}
\newunicodechar{𝒞}{\mathcal{C}}
\newunicodechar{𝒟}{\mathcal{D}}
\newunicodechar{ℰ}{\mathcal{E}}
\newunicodechar{ℱ}{\mathcal{F}}
\newunicodechar{𝒢}{\mathcal{G}}
\newunicodechar{ℋ}{\mathcal{H}}
\newunicodechar{ℐ}{\mathcal{I}}
\newunicodechar{𝒥}{\mathcal{J}}
\newunicodechar{𝒦}{\mathcal{K}}
\newunicodechar{ℒ}{\mathcal{L}}
\newunicodechar{ℳ}{\mathcal{M}}
\newunicodechar{𝒩}{\mathcal{N}}
\newunicodechar{𝒪}{\mathcal{O}}
\newunicodechar{𝒫}{\mathcal{P}}
\newunicodechar{𝒬}{\mathcal{Q}}
\newunicodechar{ℛ}{\mathcal{R}}
\newunicodechar{𝒮}{\mathcal{S}}
\newunicodechar{𝒯}{\mathcal{T}}
\newunicodechar{𝒰}{\mathcal{U}}
\newunicodechar{𝒱}{\mathcal{V}}
\newunicodechar{𝒲}{\mathcal{W}}
\newunicodechar{𝒳}{\mathcal{X}}
\newunicodechar{𝒴}{\mathcal{Y}}
\newunicodechar{𝒵}{\mathcal{Z}}
\newunicodechar{𝒶}{\mathcal{a}}
\newunicodechar{𝒷}{\mathcal{b}}
\newunicodechar{𝒸}{\mathcal{c}}
\newunicodechar{𝒹}{\mathcal{d}}
\newunicodechar{ℯ}{\mathcal{e}}
\newunicodechar{𝒻}{\mathcal{f}}
\newunicodechar{ℊ}{\mathcal{g}}
\newunicodechar{𝒽}{\mathcal{h}}
\newunicodechar{𝒾}{\mathcal{i}}
\newunicodechar{𝒿}{\mathcal{j}}
\newunicodechar{𝓀}{\mathcal{k}}
\newunicodechar{𝓁}{\mathcal{l}}
\newunicodechar{𝓂}{\mathcal{m}}
\newunicodechar{𝓃}{\mathcal{n}}
\newunicodechar{ℴ}{\mathcal{o}}
\newunicodechar{𝓅}{\mathcal{p}}
\newunicodechar{𝓆}{\mathcal{q}}
\newunicodechar{𝓇}{\mathcal{r}}
\newunicodechar{𝓈}{\mathcal{s}}
\newunicodechar{𝓉}{\mathcal{t}}
\newunicodechar{𝓊}{\mathcal{u}}
\newunicodechar{𝓋}{\mathcal{v}}
\newunicodechar{𝓌}{\mathcal{w}}
\newunicodechar{𝓍}{\mathcal{x}}
\newunicodechar{𝓎}{\mathcal{y}}
\newunicodechar{𝓏}{\mathcal{z}}
\usepackage{listings}
\usepackage{enumerate}
\usepackage{booktabs}
\usepackage{amsmath,amssymb,amsthm}
\usepackage{bm}
\usepackage{tikz}

\def\Bin{\mathrm{Bin}}

\title{The Probability to Hit Every Bin\texorpdfstring{\\}{ }with a Linear Number of Balls}
\titlerunning{The probability to hit every bin with a linear number of balls}

\author{Stefan Walzer}{Karlsruhe Institute of Technology}{stefan.walzer@kit.edu}{https://orcid.org/0000-0002-6477-0106}{}
\authorrunning{Stefan Walzer}
\begin{document}
\maketitle
\begin{abstract}
    Assume that $2n$ balls are thrown independently and uniformly at random into $n$ bins. We consider the unlikely event $E$ that every bin receives at least one ball, showing that $\Pr[E] = Θ(b^n)$ where $b ≈ 0.836$. Note that, due to correlations, $b$ is \emph{not} simply the probability that any single bin receives at least one ball.
    More generally, we consider the event that throwing $αn$ balls into $n$ bins results in at least $d$ balls in each bin.
\end{abstract}

\section{Introduction}

Let $n,d ∈ ℕ$, and $α ∈ ℝ$ with $α ≥ d$. Let $E$ be the event that throwing $αn$ balls into $n$ bins results in at least $d$ balls in every bin. More formally, $c₁,…,c_{αn} \sim 𝒰(\{1,…,n\})$ are independent random variables where $c_j$ denotes the bin of the $i$th ball for $1 ≤ j ≤ αn$. Then $X_i = |\{ j ∈ \{1,…,αn\} \mid c_j = i\}|$ is the load of the $i$th bin for $1 ≤ i ≤ n$ and $E = \{\min_{1 ≤ i ≤ n} X_i ≥ d\}$.

To state the main result we require a distribution $Φ(α,d)$ that is a Poisson distribution truncated to values $≥ d$ and tuned to have expectation $α$. Formally $Z \sim Φ(α,d)$ satisfies
\begin{equation}
     \Pr[Z = i] = \begin{cases}
        0 & i < d\\
        \frac{1}{ζ}\frac{λ^{i-d}}{i!}
    \end{cases}\label{eq:def-Φ}
\end{equation}
where $ζ = \sum_{i ≥ d} \frac{λ^{i-d}}{i!}$ is a normalisation factor and $λ = λ(α,d)$ is tuned such that $𝔼[Z] = α$.
\begin{theorem}\ 
    \label{thm:main}
    \begin{enumerate}[(i)]
        • If $α$ and $d$ are constants with $α > d$ then
        $\Pr[E] = Θ(b^n) \text{ where } b = \frac{α^αζ}{e^αλ^{α-d}}$.
        • For $α = d$ (not necessarily constant)
        $\Pr[E] = Θ(b^n \sqrt{dn}) \text{ where } b = \frac{d^d}{e^d d!}$.
    \end{enumerate}
\end{theorem}

In \cref{app:tables} we provide code for computing $b = b(α,d)$ and tabulate some values.

\subparagraph{Related Work and Motivation}

\def\Xmin{\check{X}}
\def\Xmax{\hat{X}}
In the same setting, let $\Xmin$ be the load of the least loaded bin and $\Xmax$ the load of the most loaded bin. A lot is known about these random variables.

For instance, if $α = 1$ then $\Xmax = \frac{\log n}{\log \log n}·(1+o(1))$ with high probability \cite{Gonnet:BallsIntoBins:1981}. More general results are found in \cite{RS:Balls-Into-Bins:1998} where $α$ may depend on $n$. There are also works on computing $\Pr[\Xmax = d]$ and $\Pr[\Xmin = d]$ exactly \cite{BCO:ExactyMultinomialMin:2019}.

Our focus on $\Xmin$ for \emph{constant} $α$ may seem strange because $\Xmin$ is zero with high probability. \Cref{thm:main} merely determines the base of the exponential function that describes the speed with which $\Pr[E] = \Pr[\Xmin ≥ d]$ converges to zero for $n → ∞$.

The author stumbled upon this problem in the context of minimal perfect hash functions (a randomised data structure). The probability $\Pr[\Xmin = \Xmax = α]$ for $α ∈ ℕ$ appears in space-lower bounds for minimal $α$-perfect hash functions. The more difficult case of $\Pr[\Xmin ≥ 1]$ for $α = 2$ was useful for analysing an improved minimal perfect hash function based on cuckoo hashing \cite{LSW:ShockHash:2024}. 
Given that balls-into-bins problems pop up in many places, the author beliefs that others might find the result useful.

\section{Simple Considerations}

\subparagraph{An upper bound.}
Let $E_i$ for $1 ≤ i ≤ n$ be the event that the $i$th bin is non-empty. Since $X_i \sim \Bin(n,\frac 1n)$ we have
$\Pr[E_i] = 1-\Pr[X_i = 0] = 1 - (1-\frac 1n)^{αn} \stackrel{n → ∞}{\longrightarrow} 1-e^α$.

This suggests, falsely, that $\Pr[E] = \Pr[\bigcap_{i=1}^n E_i] \stackrel{?!}{≈} \Pr[E₁]^n = (1-e^{-α})^n$. In truth, the events $(E_i)_{1 ≤ i ≤ n}$ are negatively associated and the relation in question is actually “$\ll$” and $1-e^{-α}$ is strictly larger than the value of $b$ attained from \cref{thm:main}.

\begin{proof}[Proof of \cref{thm:main} (ii).]
If $α = d$ then $E$ occurs if and only if every bin receives \emph{exactly} $d$ balls. The probability mass function of the multinomial distribution and Stirlings Approximation of $(dn)!$ gives
\begin{align*}
    \Pr[E] &= \Pr[(X₁,…,Xₙ) = (d,…,d)] = \frac{(dn)!}{(d!)^n} · n^{-dn}\\
    &= \frac{Θ\big((dn)^{dn}·e^{-nd}\sqrt{nd}\big)}{(d!)^n n^{dn}}
    = Θ\bigg( \Big(\frac{d^{d}}{e^{d}d!}\Big)^n \sqrt{nd} \bigg).\qedhere
\end{align*}
\end{proof}
\section{Proof of Theorem \ref{thm:main} (i): The Base of the Exponential}
\label{sec:proof}
\subparagraph{Proof idea.}
The standard technique of Poissonisation exploits that the multinomial distribution of $(X₁,…,Xₙ)$ can be attained by taking independent Poisson random variables $Y₁,…,Yₙ$ and conditioning them on $\sum_{i = 1}^n Y_i = αn$. We use Poissonisation with a twist.

The idea is illustrated in \cref{fig:proof-idea}. An outcome $(x₁,…,xₙ) ∈ ℕⁿ$  contributing to $E$ must satisfy two conditions: The sum $x₁+…+xₙ$ must be $αn$ and each $x_i$ must be at least $d$. The vector $(X₁,…,Xₙ)$ follows a multinomial distribution and automatically satisfies the sum condition, but not the minimum condition. The proof considers a sequence $Z₁,…,Z_n \sim Φ(α,d)$ of independent truncated Poisson random variables. The vector $(Z₁,…,Z_n)$ automatically satisfies the minimum condition, but not the sum condition. This amounts to a mathematically simpler way to capture the outcomes we want.

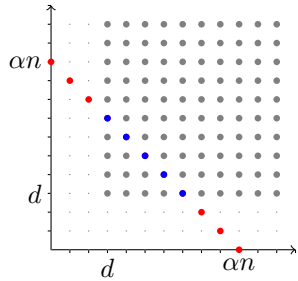
\begin{figure}[htbp]
    \begin{minipage}[c]{0.3\textwidth}
        \begin{tikzpicture}[scale=0.25]
            \def\axLen{13}
            \draw[->] (0,0) -- (\axLen,0);
            \draw[->] (0,0) -- (0,\axLen);
            \foreach \tik in {1,...,\axLen}{
                \draw (0,\tik) -- ++ (-5pt,0);
                \draw (\tik,0) -- ++ (0,-5pt);
            }
            \def\an{10}
            \def\dp{2}
            \node[left] at (0,\an) {$αn$};
            \node[below] at (\an,0) {$αn$};
            \node[left] at (0,\dp+1) {$d$};
            \node[below] at (\dp+1,0) {$d$};
            \foreach \x in {0,...,12}{
                \foreach \y in {0,...,12}{
                    \fill[gray] (\x,\y) circle (1pt);
                    \pgfmathsetmacro{\sum}{int(\x+\y)}
                    \ifnum\sum=\an
                        \fill[red] (\x,\y) circle (5pt);
                    \fi
                    \ifnum\x>\dp
                        \ifnum\y>\dp
                            \ifnum\sum=\an
                                \fill[blue] (\x,\y) circle (5pt);
                            \else
                                \fill[gray] (\x,\y) circle (5pt);
                            \fi
                        \fi
                    \fi
                }
            }
        \end{tikzpicture}
    \end{minipage}~\qquad~
    \begin{minipage}[c]{0.6\textwidth}
        \caption{Let $n = 2$, $α = 5$ and $d = 3$. The multinomial distribution $(X₁,X₂)$ automatically satisfies $X₁+X₂ = αn$ (diagonal line). A pair $(Z₁,Z₂)$ of truncated Poisson random variables automatically satisfies $Z₁ ≥ d$ and $Z₂ ≥ d$ (gray). This gives us two perspectives on the outcomes relevant for $E$ (blue), which satisfy both conditions.}
        \label{fig:proof-idea}
    \end{minipage}
\end{figure}

\begin{proof}[Proof of \cref{thm:main} (i).]
Let $R$ denote the set of possible outcomes of $\vec{X} = (X₁,…,Xₙ)$ that are consistent with $E$, meaning
\[ R = \{ \vec{x} ∈ (ℕ \setminus \{0,1,…,d-1\})^n \mid \sum_{i = 1}^n x_i = αn \}.\]
Using that $(X₁,…,Xₙ)$ has multinomial distribution gives
\begin{align}
    \Pr[E]
    &= \Pr[\vec{X} ∈ R] = \sum_{\vec{x} ∈ R} \Pr[\vec{X} = \vec{x}]
     = \sum_{\vec{x} ∈ R} \binom{αn}{x₁\,…\,xₙ} n^{-αn}\notag\\
    &= \sum_{\vec{x} ∈ R} \frac{(αn)!}{x₁!·…·xₙ!} n^{-αn}
     = \frac{(αn)!}{n^{αn}} \sum_{\vec{x} ∈ R} \frac{1}{x₁!·…·xₙ!}.\label{eq:prob-E}
\end{align}
Now consider independent $Z₁,…,Z_n \sim Φ(α,d)$ for $Φ(α,d)$ as defined in \cref{eq:def-Φ}. Let $\vec{Z} = (Z₁,…,Zₙ)$ and $N_Z = \sum_{i = 1}^n Z_i$. By construction the events $\vec{Z} ∈ R$ and $N_Z = αn$ are equivalent. For any $\vec{x} ∈ R$ we can compute
\begin{align*}
    \Pr[\vec{Z} &= \vec{x} \mid N_Z = αn]
     = \frac{\Pr[\vec{Z} = \vec{x} ∧ N_Z = αn]}{\Pr[N_Z = αn]}
     = \frac{\Pr[\vec{Z} = \vec{x}]}{\Pr[N_Z = αn]}
     = \frac{\prod_{i = 1}^n \Pr[Z_i = x_i]}{\Pr[N_Z = αn]}\\
    &= \frac{\prod_{i = 1}^n \frac{1}{ζ} · \frac{λ^{x_i -d}}{x_i!}}{\Pr[N_Z = αn]}
     = \frac{λ^{αn-dn}}{ζ^n\Pr[N_Z = αn]} \frac{1}{x₁!·…·xₙ!}.
\end{align*}
By summing this equation over all $\vec{x} ∈ R$ we get
\[
    1 = \frac{λ^{αn-dn}}{ζ^n\Pr[N_Z = αn]} \sum_{\vec{x} ∈ R}\frac{1}{x₁!·…·xₙ!}
\]
We rearrange this equation for $\sum_{\vec{x} ∈ R}\frac{1}{x₁!·…·xₙ!}$ and plug the result into \cref{eq:prob-E}. We now assume that $α$ is constant, we use Stirling's approximation of $(αn)!$ and we use that $\Pr[N_Z = αn] = Θ(1/\sqrt{n})$, which we prove in \cref{lem:prob-N_Z=αn}. This gives
\begin{align*}
    \Pr[E] &= \frac{(αn)!}{n^{αn}} \frac{ζ^n\Pr[N_Z = αn]}{λ^{αn-dn}}
    = \frac{(αn)^{αn}e^{-αn}Θ(\sqrt{n})ζ^n Θ(1/\sqrt{n})}{n^{αn} λ^{αn-dn}}
    = \Big(\frac{α^{α}ζ}{e^{α}λ^{α-d}}\Big)^n·Θ(1).
\end{align*}
This concludes the proof of \cref{thm:main}, except for the proof of \cref{lem:prob-N_Z=αn} given below.
\end{proof}

\section{Proof of Lemma \ref{lem:prob-N_Z=αn} using Log-Concavity}
\def\hp{\mathbf{\hat{\text{$p$}}}}
\def\hq{\mathbf{\hat{\text{$q$}}}}
\def\hi{\mathbf{\hat{\text{\itshape\kern-0.5pt\i\kern0.5pt}}}}

A distribution and its probability mass function (pmf) $(p_i)_{i ∈ ℤ}$ is \emph{log-concave} \cite{SW:Log-Concavity-Review:2014} if its support $\{i ∈ ℤ \mid p_i > 0\}$ is connected and $p_i² ≥ p_{i-1}·p_{i+1}$ for all $i ∈ ℤ$. The intuition, which is valid if $p_i > 0$ for all $i ∈ ℤ$, is that $i ↦ \log(p_i)$ is a concave function, meaning its discrete derivative $\log(p_{i+1}) - \log(p_i) = \log(p_{i+1} / p_i)$ is non-increasing, i.e.\ $p_i / p_{i-1} ≥ p_{i+1} / p_i$.

An example is the Poisson distribution with parameter $λ$ since its support is $ℕ₀$ and for $i ∈ ℕ$ the quotient $p_i / p_{i-1} = λ / i$ is decreasing. The truncated Poisson distribution $Φ(α,d)$ from \cref{eq:def-Φ} inherits this property. This is useful because:

\begin{lemma}[{\cite[Theorem 4.1]{SW:Log-Concavity-Review:2014}}]
    \label{lem:log-concavity-preserved}
    Log-concavity is preserved under convolution.
\end{lemma}
As in \cref{sec:proof} let $Z₁,…,Z_n \sim Φ(α,d)$ and $N_Z = \sum_{i = 1}^n Z_i$.
\begin{corollary}
    \label{cor:N_Z-log-concave}
    The distribution of $N_Z$ is log-concave.
\end{corollary}
\begin{proof}
    The pmf of $N_Z$ arises as an $n$-fold convolution of the pmf of $Φ(α,d)$, which is log-concave. Hence \cref{lem:log-concavity-preserved} applies.
\end{proof}
For the rest of this section, assume $(p_i)_{i∈ℤ}$ is a log-concave pmf, $\hp = \max_{i ∈ ℤ} p_i$ is the peak probability, $μ$ the expectation\footnotemark\addtocounter{footnote}{-1}, $σ²$ the variance\footnote{Guaranteed to exist for log-concave distributions.} and $p_μ = \max \{p_{⌊μ⌋}, p_{⌈μ⌉}\}$. If $μ ∈ ℤ$ then $p_μ$ is the probability that \emph{exactly} the expectation is attained.

We need two Lemmas regarding $\hp$ and $p_μ$ for log-concave distributions, the first of which we import from the literature.

\begin{lemma}[{\cite[Theorem 1.1]{BMM:Concentration-LogConcave:2021}}]
    \label{lem:hp-large}
    $\hp = Θ(1/(1+σ))$.
\end{lemma}
\begin{lemma}
    \label{lem:hp-vs-pμ}
    $\frac{\hp}{e} < p_μ ≤ \hp$.
\end{lemma}
\begin{proof}[Proof of \cref{lem:hp-vs-pμ}.]
    The inequality $p_μ ≤ \hp$ is true by definition. We have to show $\frac{p_μ}{\hp} > 1/e$. We may assume without loss of generality that $p_μ < \hp$ (otherwise we are done), that $\hp = p_{\hi}$ for some $\hi < μ$ (if $\hi > μ$ just mirror the setup) and that $μ ∈ [0,1)$ (otherwise shift the setup). Consider the illustration in \cref{fig:log-concave-pμ}.

    \begin{figure}[htbp]
        \begin{minipage}{0.45\textwidth}
            \includegraphics{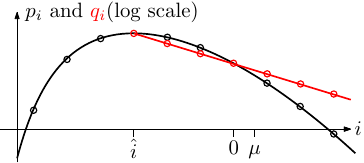}
        \end{minipage}~
        \begin{minipage}{0.5\textwidth}
            \caption{Some log-concave pmf $(p_i)_{i ∈ ℤ}$ (black) and a modified function $(q_i)_{i ∈ ℤ}$ (red) that leads, after normalisation, to a pmf where the ratio of $p_μ$ and $\hp$ is smaller.}
            \label{fig:log-concave-pμ}
        \end{minipage}
    \end{figure}
    The numbers $(q_i)_{i ∈ ℤ}$ are defined as
    \[q_i = \begin{cases}
        0 & i < \hi\\
        \hp\big(\frac{p₀}{\hp}\big)^{\frac{i - \hi}{-\hi}} & i ≥ \hi
    \end{cases}\]
    The values $q_i$ for $i ≥ \hi$ form a geometrically decreasing sequence and appear in the logarithmic plot as a straight line through $(\hi,p_{\hi})$ and $(0,p₀)$. The zero values for $i < \hi$ cannot be shown. Since we have decreased values for negative $i$ and increased values for positive $i$ we know
    \[[0,1) \ni μ = \sum_{i ∈ ℤ} i p_i ≤ \sum_{i ∈ ℤ} i q_i.\]
    By normalising $(q_i)_{i ∈ ℤ}$ we obtain a pmf $(q_i')_{i ∈ ℤ}$ with expectation $μ'$ of the same sign as $μ$, hence $μ' ≥ 0$. By construction and monotonicity we have
    \[ \frac{q_{μ'}'}{\hq'}
    = \frac{q_{μ'}'}{q'_{\hi}}
    ≤ \frac{q₀'}{q'_{\hi}}
    = \frac{q₀}{q_{\hi}}
    = \frac{p₀}{p_{\hi}}
    = \frac{p_μ}{\hp}.
    \]
    In this sense the (shifted) geometric distribution $(q'_i)_{i ∈ ℤ}$ at least as extreme an example as $(p_i)_{i ∈ℤ}$ so it is without loss of generality when we assume that $(p_i)_{i ∈ ℤ}$ is a geometric distribution to begin with (not shifted from now on for clarity). Let $λ ∈ (0,∞)$ be its parameter. We then have $p_i = 0$ for $i ≤ 0$ and $p_i = (1-λ)^{i-1}λ$ for $i > 0$. This gives $\hp = p₁ = λ$ and $μ = 1/λ$. Moreover
    \[
          \frac{p_μ}{\hp}
        = \frac{p_{⌊μ⌋}}{λ}
        = \frac{(1-λ)^{⌊μ⌋-1}λ}{λ}
        ≥ (1-λ)^{μ-1}
        ≥ (1-λ)^{1/λ-1}.
    \]
    Basic calculus shows that the function $f(λ) = (1-λ)^{1/λ-1}$ is strictly monotonic in $λ$ on $(0,1)$ with $\lim_{λ ↓ 0} f(λ) = 1/e$ and $\lim_{λ ↑ 1} f(λ) = 1$. In particular $p_μ / \hp > 1/e$ as desired.
\end{proof}

We can finally proof the lemma needed in the main theorem.
\begin{lemma}
    \label{lem:prob-N_Z=αn}
    If $α$ and $λ$ are viewed as constants with $α > λ$ then
    $\Pr[N_Z = αn] = Θ(1/\sqrt{n})$.
\end{lemma}
\begin{proof}
    Since $N_Z$ has a log-concave pmf $(p_i)_{i∈ℤ}$ by \cref{cor:N_Z-log-concave} we can apply the previous two lemmas. We also use $μ = αn$ and $σ² = \mathrm{Var}(N_Z) = \sum_{i = 1}^n \mathrm{Var}(Z_i) = n·\mathrm{Var}(Z₁) = Θ(n)$.
    \def\refrel#1#2{\stackrel{\text{Lem.\ref{#1}}}{#2}}
    \[ \Pr[N_Z = αn] = p_μ \refrel{lem:hp-vs-pμ}= Θ(\hp)
    \refrel{lem:hp-large}= Θ(1/(1+σ)) = Θ(1/(1+\sqrt{n})) = Θ(1/\sqrt{n}).\qedhere\]
\end{proof}

\bibliography{bibliography}

\appendix

\section{Sagemath code and tabulated values}
\label{app:tables}
\lstset{
    frame=single,
    backgroundcolor=,
    frameround=fttt, 
    framesep=5pt, 
}

\begin{lstlisting}[mathescape=true,frame=single,caption={Sagemath code for computing $b = b(α,d)$.},captionpos=b]
if $α$ < d:
    b = NaN
elif $α$ == d:
    b = d**d/e**d/factorial(d)
else:
    f($λ$) = e**$λ$-sum([$λ$**i/factorial(i) for i in range(d)])
    $λ$ = find_root($λ$+$λ$**d/factorial(d-1)/f($λ$)==$α$,0,$α$) #$𝔼$[Z]=$α$
    $ζ$ = $λ$**-d*f($λ$)
    b = $α$**$α$*$ζ$/e**$α$/$λ$**($α$-d)
\end{lstlisting}

\begin{table}[htbp]
\begin{tabular}{cccccc}
    \toprule
        $α$ \textbackslash\ $d$ & 1 & 2 & 3 & 4 & 5\\
    \midrule
    1 &
    0.3679 &
    - &
    - &
    - &
    - \\ 2 &
    0.8359 &
    0.2707 &
    - &
    - &
    - \\ 3 &
    0.9457 &
    0.7351 &
    0.2240 &
    - &
    - \\ 4 &
    0.9810 &
    0.8933 &
    0.6648 &
    0.1954 &
    - \\ 5 &
    0.9931 &
    0.9562 &
    0.8472 &
    0.6119 &
    0.1755\\
    \bottomrule
\end{tabular}
    \caption{Approximate values of $b = b(α,d)$ for some pairs $(α,d)$. Note that despite the selection here, non-integer values of $α$ are in principle permitted.}
\end{table}

\end{document}